\documentclass[11pt]{article}
\usepackage[ansinew]{inputenc}
\usepackage{array}
\usepackage{color}
\usepackage{amsmath,amsthm,amscd}
\usepackage{amsxtra}
\usepackage{amstext}
\usepackage{amssymb}
\usepackage{latexsym}
\usepackage{amsfonts,cite}
\usepackage{graphics,epstopdf}
\usepackage{epsfig}
\usepackage{amsfonts}
\usepackage{graphicx}
\usepackage{tcolorbox}
\usepackage{array,tabularx,colortbl}
\usepackage[framemethod=TikZ]{mdframed}

\providecommand{\U}[1]{\protect\rule{.1in}{.1in}}

\begin{document}
	
	\sloppy
	\newtheorem{thm}{Theorem}
	\newtheorem{example}{Example}
	\newtheorem{cor}{Corollary}
	\newtheorem{lem}{Lemma}
	\newtheorem{prop}{Proposition}
	\newtheorem{eg}{Example}
	\newtheorem{obs}{Observation}
	\newtheorem{defn}{Definition}
	\newtheorem{rem}{Remark}
	\numberwithin{equation}{section}
	
	\thispagestyle{empty}
	\parindent=0mm
	
	\begin{center}
		{\Large \textbf{Unveiling new perspectives of hypergeometric functions\\
				\vspace{0.2cm}
				 using umbral techniques}}
		\vspace{.20cm}\\~~	
		{\bf Giuseppe Dattoli$^{1}$, Mehnaz Haneef$^{2}$, Subuhi Khan$^{3}$ and  Silvia Licciardi$^{4}$}\\
		\vspace{0.15cm}
		
	$^{1}$ \quad ENEA---Frascati Research Center, Via Enrico Fermi 45, 00044 Rome, Italy;  \\
$^{2,3}$ \quad Department of Mathematics, Aligarh Muslim University, Aligarh, India;  \\
$^{4}$ \quad University of Palermo, Department of Engineering, Palermo, Italy. \\
	\footnote{$^{1}$E-mail:~pinodattoli@libero.it (Giuseppe Dattoli)}
	\footnote{$^{2}$E-mail:~mehnaz272@gmail.com (Mehnaz Haneef)}
	\footnote{$^{3}$E-mail:~subuhi2006@gmail.com (Subuhi Khan)(Corresponding author)}
	\footnote{$^{4}$E-mail:~silviakant@gmail.com (Silvia Licciardi)}
		\end{center}
	
	\begin{abstract}
		\noindent The umbral restyling of hypergeometric functions is shown to be a useful and efficient approach in simplifying the associated computational technicalities. In this article, the authors provide a general introduction to the umbral version of Gauss hypergeometric functions and extend the formalism to certain generalized forms of these functions. It is shown that suggested approach is particularly efficient for evaluating integrals involving hypergeometric functions and their combination with other special functions.
	\end{abstract}
	\parindent=0mm
	\vspace{.25cm}
	
	
	\vspace{0.25cm}

	\section{Introduction}
	The employment of umbral approach has been demonstrated to be a useful tool for dealing with the unified handling of special and non-special functions \cite{SLicciardi, Germano, ML, our, our2}. The philosophy that underpins umbral formalism is summarised as follows: a collection of operators is designed to introduce the image function of another function with unknown attributes, known as the object function. The image function is chosen in such a way that the relevant (known) properties are used to determine the counterpart attributes of the object function. The aforementioned objective is achieved by employing the principle of permanence of formal properties proposed in \cite{OpervsUM} and previously discussed in \cite{Peacock,Peacock2}.\\
	
	In this article, we apply the method to the theory of hypergeometric functions. The Gauss hypergeometric function is defined by the series \cite{Andrews}
	\begin{equation}\label{eq1}
		{_2{F}}_1[a,b;c;x]=\sum_{r=0}^{\infty}\frac{(a)_r(b)_r}{(c)_r}\frac{x^r}{r!},
	\end{equation} 
	where $(d)_r$ denotes the Pochhammer symbol defined by
	\begin{equation}\label{eq2}
		(d)_r=\frac{\Gamma(d+r)}{\Gamma(d)}
    \end{equation}
	and satisfies the umbral type identities, as for example \cite{Poch}:
	\begin{equation}\label{eq3}
		(k+d)_r=\sum_{s=0}^{r}\binom{r}{s}(d)_{r-s}\,(k)_s.
	\end{equation}
The umbral proof of equation \eqref{eq3} is easily achieved.
We remind that, according to the prescription considered in \cite{SLicciardi}, the following
definition can be adopted:
\begin{equation}\label{eq3a}
	\hat{s}^r\,\varphi_0=(s)_r\,,
\end{equation}
where $\hat{s}$ is an umbral operator and $\varphi_0$ is the corresponding vacuum. We can cast equation \eqref{eq3} in the form of an ordinary binomial as follows:
\begin{equation}\label{eq3b}
	(\hat{k}+\hat{d})^r\,\varphi_{0,k}\,\varphi_{0,d}=\sum_{s=0}^{r}\binom{r}{s}\hat{d}^{r-s}\hat{k}^s\,\varphi_{0,k}\,\varphi_{0,d}\,,
\end{equation}
which in view of equation \eqref{eq3a} yields equation \eqref{eq3}.\\

The Gauss hypergeometric family is characterized by the left and right indices $(p,q)$, which show the number of Pochhammer terms in the numerator and denominator, respectively. These indices specify the order of hypergeometric functions and consequently higher and lower order families can be formed. Accordingly, the hypergeometric function is reformed by making use of umbral operators,
which allow a convenient definition of its image function, realized through the following exponential function:
        \begin{equation}\label{eq4}
		{_2{F}}_1[a,b;c;x]=e^{{_2\hat{\chi}}_1 x}\phi_0,
		\end{equation}
where ${_2\hat{\chi}}_1 $ is an ad hoc umbral operator introduced to realize an elementary image function, which acts on the umbral vacuum $\phi_0$ such that
\begin{equation}\label{eq5}
{{_2\hat{\chi}}_1}^r\phi_0=\frac{(a)_r(b)_r}{(c)_r}
\end{equation}
and satisfies the following property:
\begin{equation}\label{eq6}
		{{_2\hat{\chi}}_1}^r{{_2\hat{\chi}}_1}^s\,\phi_0	= {{_2\hat{\chi}}_1}^{r+s}\,\phi_0=\frac{(a)_{r+s}(b)_{r+s}}{(c)_{r+s}}.
\end{equation}	

	\begin{obs}
	The umbral notation used in equation \eqref{eq4}, needs a clarification. The umbral operator ${}_2\hat{\chi}_1$ 
	is the shift operator such that
	\begin{equation}\label{4a}
		{_2\hat{\chi}}_1 = \frac{e^{\partial_{z_1}}\,e^{\partial_{z_2}}}{e^{\partial_{z_3}}},
	\end{equation}
	with $z_1,z_2,z_3$ as the domain's variables of the function on which the operator acts. The function $\phi_0$ is called ``vacuum" and should be specified by the following function:
		\begin{equation}\label{4b}
\phi(a,z_1, b, z_2,c ,z_3)=\frac{(a)_{z_1}\,(b)_{z_2}}{(c)_{z_3}}.	
		\end{equation}
	Finally, the action of the umbral operator on the vacuum is specified by
	\begin{equation}\label{4c}
	\frac{e^{n \partial_{z_1}}\,e^{n \partial_{z_2}}}{e^{n \partial_{z_3}}}\,\phi(a,z_1, b, z_2,c ,z_3)|_{z_1=z_2=z_3=0}=\frac{(a)_n\,(b)_n}{(c)_n}.
\end{equation}
\end{obs}
\begin{rem}
We simplify the notations, as reported below
\begin{align}\label{eq4d}
	{_2\hat{\chi}}_1 &\rightarrow \hat{\chi},\\
	\phi(a,z_1, b, z_2,c ,z_3)|_{z_1=z_2=z_3=0}	&\rightarrow \phi_0.
\end{align}
\end{rem}
\noindent In view of the above procedure, the exponential in equation \eqref{eq4} is expanded as
	\begin{equation}\label{eq7}
		e^{\hat{\chi}x}\phi_0=\sum_{r=0}^{\infty}\frac{{\hat{\chi}}^r}{r!}\,x^r \phi_0=\sum_{r=0}^{\infty}\frac{(a)_r(b)_r}{(c)_r}\frac{x^r}{r!}.
	\end{equation}
It is also evident that the identities of type \eqref{eq3} are easily framed within the umbral formalism leading to identity \eqref{eq7}.
The hypergeometric function is therefore restyled as an elementary exponential function. \\

The analysis of the characteristics of  ${_2{F}}_1[a,b;c;x]$ under the signs of derivation and integration is an intriguing application of the approach outlined in the introductory session. After these few opening observations, we consider a more comprehensive examination of applications of the method in forthcoming sections.

\section{Integrals involving hypergeometric and Gaussian functions}	
Before entering the main part of this section and in order to understand the framework of the formalism, following examples are considered:

\begin{example}
Taking the derivative of equation \eqref{eq4} w.r.t. $x$ while treating $\hat{\chi}$ as an ordinary algebraic entity, we find
\begin{equation}\label{eq8}
	\frac{d}{dx}\,	{_2{F}}_1[a,b;c;x]=\hat{\chi}e^{\hat{\chi}x}\phi_0=\hat{\chi}\sum_{r=0}^{\infty}\frac{{\hat{\chi}}^r}{r!}\,x^r \phi_0=\sum_{r=0}^{\infty}\frac{(a)_{r+1}(b)_{r+1}}{(c)_{r+1}}\frac{x^r}{r!},
\end{equation}
which, in view of the identity
\begin{equation}\label{eq9}
	(d)_{n+m}=(d)_m\,(d+m)_n,
\end{equation}
yields the well known recurrence formula
\begin{equation}\label{eq10}
	\frac{d}{dx}\,	{_2{F}}_1[a,b;c;x]=\frac{ab}{c}\,{_2{F}}_1(a+1,b+1;c+1;x).
\end{equation}	
Similarly, by using the following identity:
\begin{equation}\label{eq11}
	(d)_{-r}=\frac{(-1)^r}{(1-d)_r},
\end{equation}
it follows that
\begin{equation*}
	\int \, {_2{F}}_1[a,b;c;x]\,dx = \int e^{\hat{\chi}x}dx\;\phi_0 = {\hat{\chi}}^{-1}e^{\hat{\chi}x}\phi_0.
\end{equation*}	
Consequently, the following integral is obtained:
\begin{equation}\label{eq12}
	\int \, {_2{F}}_1[a,b;c;x]\,dx =\frac{(c-1)}{(a-1)(b-1)}\,{_2{F}}_1[a-1,b-1;c-1;x].
\end{equation}
\end{example}

In order to demonstrate the validity of the umbral formalism, we explore another well known, but less trivial example. 

\begin{example}
Consider the evaluation of following integral \cite{Abramowitz}:
\begin{equation}\label{eq13}
	I(a,b;c;\nu)=\int_{0}^{\infty}t^{\nu -1}\,{_2{F}}_1[a,b;c;-t]\,dt.
\end{equation}	
	In view of equation \eqref{eq4}, the r.h.s. of above equation can be expressed in the following form:
	\begin{equation}\label{eq14}
	\int_{0}^{\infty}t^{\nu -1}\,{_2{F}}_1[a,b;c;-t]\,dt=	\int_{0}^{\infty}t^{\nu -1}e^{-\hat{\chi}t}dt\;\phi_0, \quad 0 < \textrm{Re}(\nu) < \textrm{min}({\textrm{Re}(a), \textrm{Re}(b)}).
	\end{equation}
	Making use of integral representation of Euler Gamma function in the r.h.s. of equation \eqref{eq14}, we have
		\begin{equation}\label{eq15}
		\int_{0}^{\infty}t^{\nu -1}e^{-\hat{\chi}t}dt\,\phi_0=\Gamma{(\nu)}{\hat{\chi}}^{-\nu}\phi_0,
	\end{equation}
	which on using equation \eqref{eq5}, yields
	\begin{equation}\label{eq16}
	I(a,b;c;\nu)=\int_{0}^{\infty}t^{\nu -1}\,{_2{F}}_1[a,b;c;-t]\,dt=\Gamma{(\nu)}	\frac{(a)_{-\nu}(b)_{-\nu}}{(c)_{-\nu}}.
	\end{equation}
The above integral is identical to the conclusion provided in reference \cite{Abramowitz}, where the criteria for convergence of the integral is also described.
\end{example}

A further case, seemingly undisclosed in the literature, is provided in the following example:

\begin{example}
Consider the evaluation of following integral:
	\begin{equation}\label{eq17}
	I(a,b;c;\mu;\nu)=	\int_{0}^{\infty}t^{\nu -1}\,{_2{F}}_1[a,b;c;-t^\mu]\,dt.
	\end{equation}	
Using equation \eqref{eq4} in the r.h.s. of above integral, we have
\begin{equation}\label{eq18}
		\int_{0}^{\infty}t^{\nu -1}\,{_2{F}}_1[a,b;c;-t^\mu]\,dt=\int_{0}^{\infty}t^{\nu -1}e^{-\hat{\chi}t^\mu}dt\,\phi_0.
\end{equation}
In view of the following identity:
\begin{equation*}
	\int_{0}^{\infty}t^{\nu -1}e^{-\alpha t^\mu}\,dt=\frac{1}{\mu}\Gamma{\left(\frac{\nu }{\mu}\right)}{\alpha}^{-\left(\frac{\nu }{\mu}\right)},
\end{equation*}	
the r.h.s. of equation \eqref{eq18} simplifies to
\begin{equation}\label{eq19}
\int_{0}^{\infty}t^{\nu -1}e^{-\hat{\chi}t^\mu}dt\;\phi_0=\frac{1}{\mu}\Gamma{\left(\frac{\nu }{\mu}\right)}{\hat{\chi}}^{-\left(\frac{\nu }{\mu}\right)}\phi_0.
\end{equation}
Finally, using equation \eqref{eq5} in the r.h.s. of above equation, we get
\begin{equation}\label{eq20}
I(a,b;c;\mu;\nu)
=	\int_{0}^{\infty}t^{\nu -1}\,{_2{F}}_1[a,b;c;-t^\mu]\,dt=
\frac{1}{\mu}\Gamma{\left(\frac{\nu }{\mu}\right)}\frac{(a)_{-\left(\frac{\nu }{\mu}\right)} \, (b)_{-\left(\frac{\nu }{\mu}\right)}}{(c)_{-\left(\frac{\nu }{\mu}\right)}}.
\end{equation}
\end{example}

As further illustrations of the umbral formalism, we consider the following extension of integral given in equation \eqref{eq12}:

	\begin{equation}\label{eq21}
		\int x^\alpha\,{_2{F}}_1[a,b;c;x]\,dx =\int x^\alpha e^{\hat{\chi}x}\phi_0\,dx=x^{\alpha +1}\sum_{r=0}^{\infty}\frac{{\hat{\chi}}^rx^r}{r!(\alpha +r +1)} \phi_0.
	\end{equation}
Making use of the identity
\begin{equation}\label{eq22}
	\frac{1}{(\alpha+r+1)}=\frac{\Gamma{(\alpha+r+1)}}{\Gamma{(\alpha+r+2)}}=\frac{\Gamma{(\alpha+1)}(\alpha+1)_r}{\Gamma{(\alpha+2)}(\alpha+2)_r}.
\end{equation}	
Simplifying the r.h.s. of equation \eqref{eq21}, we have
	\begin{equation}\label{eq23}
		\int x^\alpha\,{_2{F}}_1[a,b;c;x]\,dx =\frac{x^{\alpha+1}}{\alpha+1}\sum_{r=0}^{\infty}\frac{x^r(\alpha+1)_r}{r!\,(\alpha+2)_r}{\hat{\chi}}^r\phi_0,
	\end{equation}
which, in view of equation \eqref{eq5}, gives
		\begin{equation}\label{eq24}
		\int x^\alpha\,{_2{F}}_1[a,b;c;x]\,dx =\frac{x^{\alpha+1}}{\alpha+1}\sum_{r=0}^{\infty}\dfrac{x^r}{r!}\frac{(a)_r\,(b)_r\, (\alpha+1)_r}{(c)_r\,(\alpha+2)_r}.
	\end{equation}
	Consequently, we get
		\begin{equation}\label{eq25}
		\int x^\alpha\,{_2{F}}_1[a,b;c;x]\,dx =\frac{x^{\alpha+1}}{\alpha+1}\,{_3{F}}_2[a,b,\alpha+1;c,\alpha+2;x].
		\end{equation}
	
	Taking into account the result obtained in equation \eqref{eq25}, we conclude the following important features:\\
	
	{\bf (i)} The hypergeometric function in r.h.s. of equation \eqref{eq25} has left and right indices $(3,2)$ that are distinct from those that characterize the Gauss forms. Hence, we introduce the $(p,q)$ extension, which is described by the following series:
	\begin{equation}\label{eq26}
		{}_p{F}_q[a_1,a_2,...,a_p;b_1,b_2,...,b_q;x]=\sum_{r=0}^{\infty}\frac{[a_p]_r}{[b_q]_r}\frac{x^r}{r!},
	\end{equation}
	where
	\begin{equation}\label{eq27}
		[d_m]_r=\prod\limits_{k=1}^{m}(d_k)_r\,,
	\end{equation}
having non-zero radius of convergence for $p>q+1.$\\
	
{\bf (ii)} The hypergeometric functions with left and right indices $(0,1), (1,0)$ and $(1,1)$ can also be defined accordingly.\\

{\bf (iii)} The image function of hypergeometric function of higher order can be realized by a lower order geometric function, for example
\begin{equation}\label{eq28}
	{_3{F}}_2[a,b,\alpha;c,\alpha+1;x]={_1{F}}_1[\alpha;\alpha+1;\hat{\chi}x]\phi_0.
\end{equation}
This is also intriguing as it provides an additional extension of the proposed formalism.\\

In order to provide further extension, let us now consider the following integral:
\begin{equation}\label{eq29}
	\int x^\alpha \, e^{-x}\,{_2{F}}_1[a,b;c;x]\,dx =	\int x^\alpha \,e^{-x} \,e^{\hat{\chi}x}\,dx\,\phi_0, \qquad 0<x<1.
\end{equation}

Expanding the r.h.s. of equation \eqref{eq29} and simplifying, it follows that:
\begin{equation}\label{eq29a}
\begin{split}
\int x^\alpha \,e^{-x} \,e^{\hat{\chi}x}dx\;\phi_0&=\sum_{r=0}^{\infty}\frac{(-1)^r}{r!}\int x^{\alpha + r}(1-\hat{\chi})^rdx\;\phi_0\\
&=\frac{x^{\alpha+1}}{\alpha + 1}\sum_{r=0}^{\infty}\biggl((-1)^r\frac{(\alpha + 1)_r\,x^r}{(\alpha + 2)_r\,r!}\biggr)\sum_{s=0}^{r}(-1)^s\binom{r}{s}\hat{\chi}^s\phi_0\\
	&=\frac{x^{\alpha+1}}{\alpha+1}\sum_{r=0}^{\infty}(-1)^r\frac{(\alpha + 1)_r\,x^r}{(\alpha + 2)_r\,r!}\sum_{s=0}^{r}(-1)^s\binom{r}{s}\frac{(a)_s\,(b)_s}{(c)_s}, 
	\end{split}
\end{equation}
which in view of the identity
\begin{equation}
	(-r)_s=\frac{(-1)^s\,r!}{(r-s)!}, \qquad r > s,
\end{equation}
eventually yields 
	\begin{equation}\label{eq29b}
	\int x^\alpha \,e^{-dx} \,e^{\hat{\chi}x}\phi_0\,dx	=\frac{x^{\alpha+1}}{\alpha+1}\sum_{r=0}^{\infty}\frac{(\alpha + 1)_r\,x^r}{(\alpha + 2)_r\,r!}\,{_3{F}}_1[-r,a,b;c],\quad  0<x<1.		
\end{equation}

It is interesting to note that the proposed formalism leads to several other new and fascinating identities. We consider the following result:
\begin{thm}
	For the function ${_1{F}_2}[a;b,c;\beta x]$, the following integral representation holds:
		\begin{equation}\label{eq40}
		\int_{-\infty}^{\infty}e^{-\alpha x^2}\,{_1{F}_2}[a;b,c;\beta x]\,dx=		\sqrt{\frac{\pi}{\alpha}}\,{_2{F}}_4\left[\frac{a}{2},\frac{a+1}{2};\frac{b}{2},\frac{b+1}{2},\frac{c}{2},\frac{c+1}{2};\left(\frac{\beta^2}{16\alpha}\right)\right].
	\end{equation}	
\end{thm}
\begin{proof}
Let us introduce a slightly different umbral operator with its related vacuum in the following form:
\begin{equation}\label{key6}
{}_1\hat{\chi}_2^n \;\psi_0=\dfrac{(a)_n}{(b)_n(c)_n}.
\end{equation}
Proceeding as in Section 1, we can write
\begin{equation*}
	\int_{-\infty}^{\infty}e^{-\alpha x^2}\,{_1{F}_2}[a;b,c;\beta x]\,dx=	\int_{-\infty}^{\infty}e^{-\alpha x^2}e^{\beta\;{}_1\hat{\chi}_2\; x}dx\;\psi_0 = \sqrt{\frac{\pi}{\alpha}}\,
	e^{\frac{\beta^2\;{{}_1\hat{\chi}_2}^2}{4\alpha}}\,\psi_0,
\end{equation*}	
	so that we have
	\begin{equation}\label{eq34}
\int_{-\infty}^{\infty}e^{-\alpha x^2}\,{_1{F}_2}[a;b,c;\beta x]\,dx=	\sqrt{\frac{\pi}{\alpha}}\,\sum_{r=0}^{\infty}\frac{1}{r!}\left(\frac{\beta^2}{4\alpha}\right)^r{{}_1\hat{\chi}_2}^{2r}\,\psi_0.
\end{equation}
In view of equation \eqref{eq5}, we have
\begin{equation}\label{eq38d}
		\int_{-\infty}^{\infty}e^{-\alpha x^2}\,{_1{F}_2}[a;b,c;\beta x]\,dx =\sqrt{\frac{\pi}{\alpha}}\,\sum_{r=0}^{\infty}\frac{(a)_{2r}}{(b)_{2r}\,(c)_{2r}}\dfrac{1}{r!}\left(\frac{\beta^2}{4\alpha}\right)^r.
\end{equation}

	By taking into account the following identity:
	\begin{equation}\label{eq38}
		(d)_{2r}=2^{2r}\left(\frac{d}{2}\right)_r\left(\frac{d+1}{2}\right)_r,
	\end{equation}
	we find
	\begin{equation}\label{eq39}
		\int_{-\infty}^{\infty}e^{-\alpha x^2}\,{_1{F}_2}[a;b,c;\beta x]\,dx=		\sqrt{\frac{\pi}{\alpha}}\,\sum_{r=0}^{\infty}\frac{\left(\frac{a}{2}\right)_{\!r}\,\left(\frac{a+1}{2}\right)_{\!r}}{\left(\frac{b}{2}\right)_{\!r}\,\left(\frac{b+1}{2}\right)_{\!r}\,\left(\frac{c}{2}\right)_{\!r}\,\left(\frac{c+1}{2}\right)_{\!r}}\dfrac{1}{r!}\left(\frac{\beta^2}{16\alpha}\right)^r.
	\end{equation}
Finally, considering the series definition of hypergeometric function, assertion \eqref{eq40} follows.
\end{proof}
	Further, we establish an integral in terms of Fox-Wright $\Psi$-functions \cite{Andrews}. This family of functions is specified by the following series:
	\begin{equation}\label{eq70}
		{}_p{\Psi}_q\begin{pmatrix}
			(a_1,A_1) & (a_2,A_2) & ...\,(a_p,A_p)\\
			& & & ;x\\
			(b_1,B_1) & (b_2,B_2) & ...\, (b_q,B_q)
		\end{pmatrix}
		=\sum_{r=0}^{\infty}\frac{\prod\limits_{k=1}^{p}\Gamma(a_k+rA_k)\,x^r}{\prod\limits_{s=1}^{q}\Gamma(b_s+rB_s)\,r!}.
	\end{equation}	
	According to the above definition, integral \eqref{eq38d} can be expressed as:

	\begin{equation}\label{eq72}
		\int_{-\infty}^{\infty}e^{-\alpha x^2}\, {_1{F}}_2[a,b;c;\beta x]\,dx	=\sqrt{\frac{\pi}{\alpha}}\,\frac{\Gamma(b)\Gamma(c)}{\Gamma(a)}\,	{_1{\Psi}}_2\begin{pmatrix}
			& (a,2)  & \\
			& & & ;\frac{\beta^2}{4\alpha}\\
			(b,2)  &  & (c,2) 
		\end{pmatrix}.
	\end{equation}
In view of equations \eqref{eq40} and \eqref{eq72}, we also obtain the following relation:
	\begin{equation}\label{eq72d}
		{_2{F}}_4\left[\frac{a}{2},\frac{a+1}{2};\frac{b}{2},\frac{b+1}{2},\frac{c}{2},\frac{c+1}{2};\left(\frac{\beta^2}{16\alpha}\right)\right]	=\frac{\Gamma(b)\Gamma(c)}{\Gamma(a)}\,	{_1{\Psi}}_2\begin{pmatrix}
			& (a,2)  & \\
			& & & ;\frac{\beta^2}{4\alpha}\\
			(b,2)  &  & (c,2) 
		\end{pmatrix}.
	\end{equation}
	This outcome provides insight into the adaptability and generality of the formalism developed so far. \\
	
The integral corresponding to equation \eqref{eq40} for the generalized hypergeometric function ${}_p{F}_q[a_1,a_2,...,a_p;b_1,b_2,...,b_q;\beta x]$ is obtained in the form of following result:
\begin{thm}\label{t1}
	For the function ${}_p{F}_q[a_1,a_2,...,a_p;b_1,b_2,...,b_q;\beta x]$, the following integral holds:
	\begin{align*}
		\int_{-\infty}^{\infty}&e^{-\alpha x^2}{}_p{F}_q[a_1,...,a_p;b_1,...,b_q;\beta x]\,dx=\\
		&\sqrt{\frac{\pi}{\alpha}}\,\,{}_{2p}F_{2q}\left[\frac{a_1}{2},\frac{a_1+1}{2},...,\frac{a_p}{2},\frac{a_p+1}{2};\frac{b_1}{2},\frac{b_1+1}{2},...,\frac{b_q}{2},\frac{b_q+1}{2};2^{2(p-q)}\left(\frac{\beta^2}{4\alpha}\right)\right],\quad p \le q+1.
		\end{align*}
\end{thm}

A further interesting integral can be obtained in the following example:
\begin{example}
\begin{equation}\label{es4}
\int_{-\infty}^\infty {}_1F_2[a;b,c;-\alpha x^2+\beta x]dx=\sqrt{\dfrac{\pi}{\alpha}}\sum_{r=0}^\infty \dfrac{(a)_{r-\frac{1}{2}}}{(b)_{r-\frac{1}{2}}(c)_{r-\frac{1}{2}}}\dfrac{1}{r!}\left( \dfrac{\beta^2}{4\alpha}\right)^r. 
\end{equation}	
The relevant proof is achieved by setting
\begin{equation*}
\int_{-\infty}^\infty {}_1F_2[a;b,c;-\alpha x^2+\beta x]\,dx=\int_{-\infty}^\infty e^{-\alpha\;{}_1\hat{\chi}_2\;x^2 }e^{\beta\;{}_1\hat{\chi}_2\;x}dx\;\psi_0,
\end{equation*}
from which, it follows that
\begin{equation}\label{key}
\int_{-\infty}^\infty {}_1F_2[a;b,c;-\alpha x^2+\beta x]\,dx= \sqrt{\dfrac{\pi}{\alpha}}\;{}_1\hat{\chi}_2^{-\frac{1}{2}}e^{{}_1\hat{\chi}_2\;\frac{\beta^2}{4\alpha}}\psi_0.
\end{equation}
		
Expanding the operator ${}_1\hat{\chi}_2$, integral \eqref{es4} is obtained.
\end{example}

The above results can be extended by using certain special cases of hypergeometric functions. These special cases are considered in the next section.

\section{Applications}

We consider the following special cases of hypergeometric functions:\\

{\bf Case I.~Confluent hypergeometric function}\\

We focused on this particular case because the confluent hypergeometric function plays a pivotal role in physics. For example, these functions are often employed in the study of bound state problems for the simple harmonic oscillator in one, two, and three dimensions; the Coulomb problem in two and three dimensions and the cartesian one dimensional Morse potential problem. All the aforementioned problems can be solved in terms of confluent hypergeometric function (more specifically in terms of Kummer's function) \cite{Andrews}. Indeed, the radial function for Landau states can be written as \cite{Mathews}:
\begin{equation}\label{eq46}
	\mathcal{R}(\xi)=e^{-\frac{\xi}{2}}\,{\xi}^{\frac{|m_l|}{2}}\,{_1{F}}_1[a;b;\xi],
\end{equation}
where
\begin{equation}\label{eq47}
	a:=-\lambda+\frac{|m_l|+1}{2}, \qquad b:=1+|m_l|, \quad m_l\in\mathcal{R},
\end{equation}
$m_l$ is a real parameter characteristic of the physical problem.\\
Without elaborating the physical significance, it is apparent that in umbral form, we may write
\begin{equation}\label{eq48}
	\mathcal{R}(\xi)={\xi}^{\frac{|m_l|}{2}}\,e^{-\left(\frac{1}{2}-\;{}_1\hat{\chi}_1\right)\xi}\,\gamma_0, \qquad \qquad {}_1\hat{\chi}_1^n \gamma_0=\dfrac{(a)_n}{(b)_n}.
\end{equation}
In view of Theorem \ref{t1}, we have the following result for the confluent hypergeometric function:
\begin{cor}
	For the confluent hypergeometric function ${_1{F}}_1[a;b;\beta x]$, the following integral holds:
	\begin{equation}\label{eq49}
			\int_{-\infty}^{\infty}e^{-\alpha x^2}{_1{F}}_1[a;b;\beta x]\,dx=
			\sqrt{\frac{\pi}{\alpha}}\,{_{2}F}_{2}\left[\frac{a}{2},\frac{a+1}{2};\frac{b}{2},\frac{b+1}{2};\left(\frac{\beta^2}{4\alpha}\right)\right].	
	\end{equation}
\end{cor}

\vspace{0.3cm}

Integrals involving the function  $\mathcal{R}(\xi)$ are of similar type as those mentioned in equations \eqref{eq29}-\eqref{eq29b}. The umbral restyling in equation \eqref{eq48} can consequently be effective for calculating overlapping integrals between distinct Landau states.\\

{\bf Case II.~Hypergeometric representation of Bessel functions}\\

The $0$th order cylindrical Bessel function of first kind can be written in terms of hypergeometric function as \cite{Andrews}:
\begin{equation}\label{eq50}
	J_0(x)={_0{F}}_1\left(- ;1;-\frac{x^2}{4}\right)=\sum_{r=0}^{\infty}\frac{(-1)^r}{(r!)^2}\left(\frac{x}{2}\right)^{2r}.
\end{equation}
The relevant umbral representation is same as used in previous publications, see, for example \cite{SLicciardi}. Consider the $0$th order Tricomi function \cite{Tricomi}
\begin{equation}\label{eq51}
	C_0(x)=J_0(2\sqrt{x})={_0{F}}_1[-;1;-x]=\sum_{r=0}^{\infty}\frac{(-x)^r}{(r!)^2}.
\end{equation}
Accordingly, applying Theorem \ref{t1}, it follows that
\begin{equation}\label{eq52}
\int_{-\infty}^{\infty}e^{-\alpha x^2}C_0(\beta x)\,dx=\int_{-\infty}^{\infty} e^{-\alpha x^2}\,{_0{F}}_1[-;1;-\beta x]\,dx=\sqrt{\frac{\pi}{\alpha}}\,{_0{F}}_2\left[-;\frac{1}{2},1;\left(\frac{\beta^2}{16\alpha}\right)\right].
\end{equation}

 \vspace{0.30cm}
 
{\bf Case III.~Circular functions and hypergeometric functions}\\

The hypergeometric representation of cosine function is given as:
\begin{equation}\label{eq53}
	\cos(x)={_0{F}}_1\left[-;\frac{1}{2};-\frac{x^2}{4}\right].
\end{equation}
As a consequence of previously derived identities we obtain the following integral for the cosine function:
\begin{equation}\label{eq54}
	\int_{-\infty}^{\infty}e^{-\alpha x^2}\cos(\sqrt{\beta x})=\sqrt{\frac{\pi}{\alpha}}\,{_0{F}}_2\left[-;\frac{1}{4},\frac{3}{4};\left(\frac{\beta^2}{2^8\alpha}\right)\right].
\end{equation}

\vspace{0.3cm}

The above two equations are important within the umbral context developed here. In fact, prior works have addressed the problem of incorporating trigonometric functions into umbral formalism. The connection to the current formalism will be covered in the last section as concluding remarks.\\

\newpage

{\bf Case IV.~Geometric series and hypergeometric functions}\\

The hypergeometric representation for the geometric series is given by
\begin{equation}\label{eq55}
	\frac{1}{1+z^2}={_1{F}}_0[1;-;-z^2].
\end{equation}
The following umbral operator is introduced:
\begin{equation}\label{opP}
{}_1{}\hat{\chi}_0^n \; p(b)_0:={}^b\hat{\chi}^n p_0=(b)_n,
\end{equation}
so that, we can write
\begin{equation}\label{eq56}
{}_1{F}_0[1;-;-z^2]=e^{-\;{}^1\hat{\chi}\;z^2}\,p_0.
\end{equation}
According to the formalism discussed so far, the following integral is easily obtained:
\begin{equation}\label{eq58}
	\int_{-\infty}^{\infty}{_1{F}}_0[1;-;-z^2]\,dz=\sqrt{\pi}\; {}^1\hat{\chi}^{-\frac{1}{2}}\,p_0=\pi.
\end{equation}
Further, we derive the following integral:
\begin{equation}\label{key3}
	\int_{-\infty}^{\infty}{_1{F}}_0[1;-;-\alpha z^2+ \beta z]\,dz=\dfrac{\pi}{\sqrt{\alpha\left( 1-\frac{\beta^2}{4\alpha} \right) }}\,.
\end{equation}
Using the method outlined above, we have
\begin{equation*}
\int_{-\infty}^{\infty}{_1{F}}_0[1;-;-\alpha z^2+ \beta z]\,dz =
		\sqrt{\frac{\pi}{{}^1\hat{\chi}\; \alpha}} e^{{}^1\hat{\chi}\frac{\beta^2}{4\alpha}}\,p_0 = \sqrt{\dfrac{\pi}{\alpha}}\sum_{r=0}^\infty \dfrac{{}^1\hat{\chi}^{r-\frac{1}{2}}}{r!}\left(\dfrac{\beta^2}{4\alpha} \right)^r p_0,
\end{equation*}
which simplifies to
\begin{equation}\label{eq59}
\int_{-\infty}^{\infty}{_1{F}}_0[1;-;-\alpha z^2+ \beta z]	 =\sqrt{\frac{\pi}{\alpha}}\sum_{r=0}^{\infty}\frac{(1)_{r-\frac{1}{2}}}{r!}\left(\frac{\beta^2}{4\alpha}\right)^r.
\end{equation}
In view of equation \eqref{eq9}, we have
\begin{equation}\label{key1}
(1)_{r-\frac{1}{2}}=(1)_{-\frac{1}{2}}\left(1-\frac{1}{2} \right)_r=\sqrt{\pi} \left(\frac{1}{2} \right)_r.
\end{equation}
Using equation \eqref{key1} in the r.h.s. of equation \eqref{eq59}, we find
\begin{equation}\label{key2}
\int_{-\infty}^{\infty}{_1{F}}_0[1;-;-\alpha z^2+ \beta z]\,dz=\dfrac{\pi}{\sqrt{\alpha}}\;{}_1F_0\left[ \dfrac{1}{2};\;\;;\dfrac{\beta^2}{4\alpha}\right],
\end{equation}
which yields integral \eqref{key3}.\\

According to the current perspective, usage of the ideas connected with Borel transform \cite{OpervsUM} is a crucial concept that has been discussed previously in relation to the rigorous foundations of umbral calculus. Even though it requires some more work, the integral transform approach is a potent instrument for many of the findings so far.
The umbral approach combined with the integral representation of hypergoemetric function is a handy tool for obtaining further results. It is widely acknowledged that the integral representation of confluent hypergeometric function writes \cite{Andrews}

	\begin{equation}\label{eq62}
	{_1{F}}_1[a;c;x]=\frac{\Gamma(c)}{\Gamma(a)\Gamma(c-a)}\int_{0}^{1} t^{a-1} (1-t)^{c-a-1} e^{xt}\,dt.
	\end{equation}
	
	In accordance with integral \eqref{eq49}, we have
	
\begin{equation}\label{eq63}
			\int_{-\infty}^{\infty} e^{-\alpha x^2}{_1{F}}_1[a;b;\beta x]\,dx =\frac{\Gamma(b)}{\Gamma(a)\Gamma(b-a)}\int_{0}^{1} t^{a-1} (1-t)^{b-a-1} \left(\int_{-\infty}^{\infty}e^{-\alpha x^2 + \beta xt}\,dx\right)\,dt,	
		\end{equation}
			which further simplifies to
	\begin{equation}
			\int_{-\infty}^{\infty} e^{-\alpha x^2}{_1{F}}_1[a;b;\beta x]\,dx=\sqrt{\frac{\pi}{\alpha}}\frac{\Gamma(b)}{\Gamma(a)\Gamma(b-a)}\int_{0}^{1} t^{a-1} (1-t)^{b-a-1}e^{\frac{\beta^2}{4\alpha}t^2}\,dt.	
\end{equation}

	Comparison between equations \eqref{eq49} and \eqref{eq63} eventually yields
	
	\begin{equation}\label{eq63a}
		{_2{F}}_2\left[\frac{a}{2},\frac{a+1}{2};\frac{b}{2},\frac{b+1}{2};\frac{\beta^2}{4\alpha}\right]=\frac{\Gamma(b)}{\Gamma(a)\Gamma(b-a)}\int_{0}^{1} t^{a-1} (1-t)^{b-a-1}e^{\frac{\beta^2}{4\alpha}t^2}\,dt.
	\end{equation}
	
	\vspace{0.25cm}
	
	As another illustration, we establish the following integral representation for hypergeometric function:
		\begin{equation}\label{eq64}
			{_2{F}}_1[a,b;c;x]=\frac{\Gamma(c)}{\Gamma(a)\Gamma(c-a)}\int_{0}^{1} t^{a-1} (1-t)^{c-a-1}\,{_1{F}}_0[b;-;xt]\,dt.
		\end{equation}

	By using umbral operator \eqref{opP},
		we can write
		\begin{equation}\label{eq65}
			{}_2{F}_1[a,b;c;x]={}_1{F}_1[a;c;{}^b\hat{\chi}\;x]p_0,
		\end{equation}
which in view of integral \eqref{eq62} takes the form

\begin{equation}\label{eq67}
	{}_2{F}_[(a,b;c;x]=	\frac{\Gamma(c)}{\Gamma(a)\Gamma(c-a)}\int_{0}^{1} t^{a-1} (1-t)^{c-a-1} \,e^{{}^b\hat{\chi}\;xt}\,dt\,p_0.
\end{equation}
Again using equation \eqref{opP}, we find
\begin{equation}
		{}_2{F}_1[a,b;c;x]= \frac{\Gamma(c)}{\Gamma(a)\Gamma(c-a)}\sum_{r=0}^{\infty}\frac{1}{r!}\,\int_{0}^{1} t^{a-1} (1-t)^{c-a-1} (b)_r\,(xt)^r\,dt,
\end{equation}
which consequently yields integral \eqref{eq64}.\\

 In the next section, we consider the embedding of theory of hypergeometric differential equations within the umbral formalism.

\section{Umbral methods and hypergeometric differential equations}	
From earlier considerations, it may be evident that the construction of hypergeometric functions is modular. These functions may be built by adding (or removing) sections that include Pochhammer terms and by altering the punctuations accordingly, as for example ${_1{F}}_1(a;c;x) \rightarrow {_2{F}}_1(a,b;c;x)$. The structure of the corresponding differential equations ought to exhibit the same pattern.\\

In the past, the idea of Laguerre derivative \cite{ SLicciardi,Datto,Cacao} was proposed in relation to the investigation of the monomiality principle. According to this formalism, the $0$th  order Tricomi function $C_0(x)$, is an eigen function of Laguerre derivative, namely	
	
	\begin{equation}\label{eq73}
		{_L{\hat{D}}}_x\,C_0(\lambda x)=-\lambda C_0(\lambda x),
	\end{equation}
where
\begin{equation}\label{eq74}
		{_L{\hat{D}}}_x:=\frac{d}{dx}x\frac{d}{dx}.
\end{equation}	
	Keeping into account that
	\begin{equation}\label{eq75}
		{_L{\hat{D}}}_x \rightarrow \left(1+x\frac{d}{dx}\right)\frac{d}{dx},	
	\end{equation}
we find
\begin{equation}\label{eq76}
	C_0(x) \rightarrow {_0{F}}_1[;1;-x].
\end{equation}	

\vspace{0.30cm}

	Moreover, the $\nu$th order Tricomi functions are specified by
	\begin{equation}\label{eq77}
	C_{\nu}(x)=\sum_{r=0}^{\infty} \frac{(-x)^r}{r!\,\Gamma(\nu +r +1)}
	\end{equation}
	and the associated eigenvalue equation reads
	\begin{equation}\label{eq78}
		\left(\frac{d}{dx}x\frac{d}{dx}+\nu \frac{d}{dx}\right)C_\nu(\lambda x)=-\lambda C_\nu(\lambda x).
	\end{equation}
	The link between Tricomi functions and hypergeometric functions is provided by
	\begin{equation}\label{eq79}
		C_{b-1}(x)=\frac{1}{\Gamma(b)}\,{_0{F}}_1[-;b;-x].
	\end{equation}
	Eventually, equation \eqref{eq78} can be written as:
	\begin{equation}\label{eq80}
			\left(x\frac{d}{dx}+b\right)\frac{d}{dx}\,{_0{F}}_1[-;b+1;-\lambda x]=\lambda\,{_0{F}}_1[-;b+1;-\lambda x].
	\end{equation}
Using the umbral point of view, we introduce the following operator:
 \begin{equation}\label{key7}
 	{}_0\hat{\chi}_1^n\;\varepsilon(b)_0:={}_b\hat{\chi}^n\varepsilon_0=\frac{1}{(b)_n}
 \end{equation}
  and the differrential operator as follows:
 \begin{equation}\label{key5}
 \dfrac{d}{d\;{}_b\hat{\chi}\;x}:={}_b\overline{D}_x.
 \end{equation}
 Accordingly, the eigenvalue problem \eqref{eq80} can be expressed as:
	\begin{equation}\label{eq81}
{}_b\overline{D}_x\;e^{\lambda\;{}_b\hat{\chi}\;x}\varepsilon_0=\lambda 	e^{\lambda\;{}_b\hat{\chi}\;x}\varepsilon_0.
	\end{equation}
	Comparison of equations \eqref{eq80} and \eqref{eq81}, yields the following correspondence:
	\begin{equation}\label{eq84}
{}_b\overline{D}_x	\rightarrow \left(x\frac{d}{dx}+b\right)\frac{d}{dx},
		\end{equation}
	which will be used to derive the differential equation satisfied by ${_1{F}}_1[a;b;x]$. Next, we consider the following umbral form:
	\begin{equation}\label{eq85}
{_1{F}}_1[a;b;x]=e^{{}^a\hat{\chi}\;{}_b\hat{\chi}\;x}\,p_0\,\varepsilon_0,
\end{equation}
which in view of notation \eqref{key5} can be written as:
	\begin{equation}\label{eq87}
	{}_b\overline{D}_x\, e^{{}^a\hat{\chi}\;{}_b\hat{\chi}\;x}\,p_0\varepsilon_0=
{}^a\hat{\chi}\,e^{{}^a\hat{\chi}\;{}_b\hat{\chi}\;x}\,p_0\,\varepsilon_0.
\end{equation}
	In non umbral terms, we obtain
	\begin{equation}\label{eq88}
	\left(x\frac{d}{dx}+b\right)\frac{d}{dx}\,{_1{F}}_1[a;b;x]=a\,{_1{F}}_1[a+1;b;x].
	\end{equation}
Again, by noting that
	\begin{equation}\label{eq89}
		a\,{_1{F}}_1[a+1;b;x]=\left(x\frac{d}{dx}+a\right)\,{_1{F}}_1[a;b;x],
	\end{equation}
	we find
	\begin{equation}\label{eq90}
		\left(x\frac{d}{dx}-x+b\right)\frac{d}{dx}\,{_1{F}}_1[a;b;x]=a\,{_1{F}}_1[a;b;x].
	\end{equation}
	
	\vspace{0.3cm}
	
Using the same methodology, it is possible to produce the differential equations satisfied by higher order hypergeometric functions. Indeed, we get
	\begin{equation}\label{eq91}
	{}_2{F}_1(a,b;c;x)=\exp\{{}^a\hat{\chi}\;{}^b\hat{\chi}\;{}_c\hat{\chi}\;x\}\,p_{0,a}\,p_{0,b}\,\varepsilon_{0,c}
		\end{equation}
and	
\begin{equation}\label{eq92}
	{}_c\overline{D}_x\,\exp\{{}^a\hat{\chi}\;{}^b\hat{\chi}\;{}_c\hat{\chi}\;x\}\,p_{0,a}\,p_{0,b}\,\varepsilon_{0,c}=
	{}^a\hat{\chi}\;{}^b\hat{\chi}\;
\exp\{{}^a\hat{\chi}\;{}^b\hat{\chi}\;{}_c\hat{\chi}\;x\}\,p_{0,a}\,p_{0,b}\,\varepsilon_{0,c}
\end{equation}	

	\vspace{0.3cm}
	
	The above equations lead us to the following differential equation:
\begin{equation}\label{eq93}
	\left(x\frac{d}{dx}+c\right)\frac{d}{dx}\,{_2{F}}_1[a,b;c;x]=\left(x\frac{d}{dx}+a\right)\left(x\frac{d}{dx}+b\right)\,{_2{F}}_1[a,b;c;x],	
\end{equation}
	which can eventually be cast in the \textquotedblleft canonical\textquotedblright ~form of Gauss hypergeometric function as: 
\begin{equation}\label{eq94}
	\left(x\frac{d}{dx}+a\right)\left(x\frac{d}{dx}+b\right)-	\left(x\frac{d}{dx}+c\right)\frac{d}{dx} \rightarrow x(1-x)\frac{d^2}{dx^2}-[c-(a+b+1)x]\frac{d}{dx}+ab.
\end{equation}
	We have envisaged in previous sections that the umbral modular nature of hypergeometric function has impact on the relevant recurrences and differential equations. Further, we note that
	\begin{equation}\label{eq95}
	\prod\limits_{j=1}^{p}\; {}^{a_j}\hat{\chi}\;p_{0,j}= 	\prod\limits_{j=1}^{p}\left( x\frac{d}{dx} + a_j \right),
	\end{equation}
	\begin{equation}\label{eq96}
	\frac{d}{d\left(\prod\limits_{r=1}^{q}
		{}_{b_r}\hat{\chi}
		\,x\;\varepsilon_{0,r}\right)}=\prod\limits_{r=1}^{q}\left(x\frac{d}{dx}+b_r\right)\frac{d}{dx}.
	\end{equation}

Also, we note that identity \cite{Andrews}:
	\begin{equation}\label{eq97}
	\left( J_0(x)\right) ^2={}_1{F}_2\left[\frac{1}{2};1,1;-x^2\right],
	\end{equation}
can be written in umbral form as follows:
	\begin{equation}
		\left( J_0(x)\right) ^2	=
		\exp\{-\left( {}^{\frac{1}{2}}\hat{\chi}\right) \left( {}_1\hat{\chi}\right) \left( {}_{1'}\hat{\chi}\right) x^2\}\,p_0\;\varepsilon_{0,1}\;\varepsilon_{0,1'}.
	\end{equation}
	Here, the prime indicates an umbral operator acting on the vacuum $\varepsilon_{0,1'}$. The structure of prime and unprime umbral quantities is same, but considered as separate entities.\\

Consequently, the integrals involving square of Bessel functions, can be expressed as:
\begin{equation}
\int_{-\infty}^\infty e^{-\alpha x^2}\left(J_0\left(\sqrt{x} \right)  \right)^2dx=\sqrt{\dfrac{\pi}{\alpha}}
\exp\left\lbrace \dfrac{1}{4\alpha}\left[ \left( {}^{\frac{1}{2}}\hat{\chi}\right) \left( {}_1\hat{\chi}\right) \left( {}_{1'}\hat{\chi} \right) \right]  ^2\right\rbrace  \,p_0\;\varepsilon_{0,1}\;\varepsilon_{0,1'},
\end{equation}
which on simplification yields
\begin{equation}
\int_{-\infty}^\infty e^{-\alpha x^2}\left(J_0\left(\sqrt{x} \right)  \right)^2dx= \dfrac{1}{\sqrt{\alpha}}\sum_{r=0}^\infty \dfrac{\Gamma\left( \frac{1}{2}+2r\right) }{(2r)!^2r!}\left(\dfrac{1}{4\alpha} \right)^r.  
\end{equation}	
The relevant expression in terms of Fox-Wright function can be written as:
	\begin{equation}
		\int_{-\infty}^\infty e^{-\alpha x^2}\left(J_0\left(\sqrt{x} \right)  \right)^2dx=\frac{1}{\sqrt{\alpha}}	{_1{\Psi}}_2\begin{pmatrix}
			& (1/2,2)  & \\
			& & & ;\frac{1}{4\alpha}\\
			(1,2)  &  & (1,2) 
		\end{pmatrix}.
	\end{equation}
	
	In the next section, we consider extensions of the umbral approach to other functions. Certain important concluding remarks for further investigations will also be discussed.
	
	\section{Concluding remarks}
	In previous sections, we have covered different aspects of the umbral formalism and their relevance to hypergeometric functions. We have seen that umbral formalism seems to be tailored suited to simplify the associated computational technicalities. In this section, we further extend the formalism to explore the umbral hypergeometric versions of circular functions. As an illustration, we consider relation \eqref{eq53}. According to the proposed formalism , we can write the cosine differential equation as follows:
	\begin{equation}\label{eq100}
		\left(y\frac{d}{dy}+\frac{1}{2}\right)\frac{d}{dy}\,{_0{F}}_1\left[-;\frac{1}{2};y\right]={_0{F}}_1\left[-;\frac{1}{2};y\right],
	\end{equation}
	where $y:=-\frac{x^2}{4}$.
	After standard computations, equation \eqref{eq100} yields
	\begin{equation}\label{eq101}
	\frac{d^2}{dx^2}\,{_0{F}}_1\left[-;\frac{1}{2};-\frac{x^2}{4}\right]=-{_0{F}}_1\left[-;\frac{1}{2};-\frac{x^2}{4}\right].
	\end{equation}
The umbral version of the cosine function is given as \cite{SLicciardi}:
	\begin{equation}\label{eq102}
	\cos(x)=e^{-{}_{2,1}\hat{d}\;x^2}\,g_0, \qquad {}_{\alpha,\beta}\hat{d}^k g_0=\dfrac{\Gamma(k+1)}{\Gamma(\alpha k+\beta)}.
   \end{equation}
	 Accordingly, the umbral version of sine function is expressed as:
	\begin{equation}\label{eq104}
	\sin(x)=2x\; {}_{2,1}\hat{d}\; e^{-{}_{2,1}\hat{d}\;x^2}\,g_0,
	\end{equation}
	whose hypergeometric analogue writes
		\begin{equation}\label{eq105}
\sin(x)=x\,{_0{F}}_1\left[-;\frac{3}{2};-\frac{x^2}{4}\right].
\end{equation}
	The potential for creating a seamless transition between elementary and higher transcendental functions has attracted interest for the umbral restyling of circular functions. It has been established that such a context permits the replacement of trigonometric and spherical Bessel functions in a similar manner. Therefore, it is important to investigate the connection between circular and hypergeometric functions.\\
	
	Before closing the article, we underline the following two important observations:\\	
	
	{\bf (i)}~Using basic Gaussian identities, it is possible to express the integral transform of the cosine function in the following way:
\begin{equation}\label{eq106}
	\cos(x)=\frac{1}{\sqrt{\pi}}\int_{-\infty}^{\infty}e^{-\xi^2}e^{2i\sqrt{{}_{2,1}\hat{d}}\;\xi x}g_0=\frac{1}{\sqrt{\pi}}\int_{-\infty}^{\infty}e^{-\xi^2}\cos_{\frac{1}{2}}(2i\xi x)\,d\xi,
\end{equation}
where $$\cos_{\frac{1}{2}}(x)=\sum\limits_{r=0}^{\infty}\frac{\Gamma\left(\frac{r}{2}+1\right)}{(r!)^2}\,x^r,$$ whose association with Bessel like functions will be discussed later. The same result can be obtained by making the use of hypergeometric formulation.\\

{\bf (ii)}~We have discussed so far the case of single variable hypergeometric function. Their multi-variable extension is important from the point of view of applications. Even though the full discussion will be presented elsewhere, here we outline the lines for forthcoming treatment. The two-variable App\'{e}ll-Kamp\'{e}-d\'{e}-F\'{e}ri\'{e}t  functions \cite{Appel} defined by the series:
\begin{equation}\label{eq107}	
	F(\alpha;\gamma;\beta;\beta';x,y)=\sum\limits_{m,n=0}^{\infty}\frac{(\alpha)_{m+n}\,(\beta)_{m}\,x^m\,y^n}{(\gamma)_{m+n}\,(\beta')_n\,m!\,n!},
\end{equation}
can be expressed in the umbral form as:
\begin{equation}\label{eq108}		
F(\alpha;\gamma;\beta;\beta';x,y)=e^{\hat{A}\hat{B}x}e^{\hat{A}/\hat{B'}y}\,\phi_0.
\end{equation}	
	In a forthcoming investigation, we  will use and generalize the formalism developed so far to establish the umbral theory of the multi-variable hypergeometric functions. This will enable us to draw further consequences, complementing those obtained in this article, by using umbral definition of multi-variable hypergeometric functions.\\

In this article, we have shown that the umbral procedure employed to the theory of hypergeometric functions
and to the relevant generalizations offers significant advantages, which, in turn
yields non secondary advancements from the computational point of view.
Reanalyzing the properties of the Pochhammer symbol, we like to mention that
most of the results established in this article are based on the computations related to these mathematical tools. A central role is played by such type of tools in advanced methods for the evaluation of integrals involving special
functions, as for example, the method of brackets \cite{Gonzales}. In a forthcoming study, we will discuss
in details the link between the two methods along with a possible merging.\\
	
{\bf Conflict of Interest:} \\
The authors declare that they have no conflict of interest.\\

		\textbf{Acknowledgments:} \\		
Dr. Silvia Licciardi was supported by the following project:\\
\indent Funder: Project funded under the National Recovery and Resilience Plan (NRRP), Mission 4 Component 2 Investment 1.3 - Call for tender No. 1561 of 11.10.2022 of Ministero dell\'~Universit\`a  e della Ricerca (MUR); funded by the European Union -- NextGenerationEU.\\		 
		 \indent Award Number: Project code PE0000021, Concession Decree No. 1561 of 11.10.2022 adopted by Ministero dell\'~Universit\`a e della Ricerca (MUR), CUP -  Project title: Network 4 Energy Sustainable Transition -- NEST.

	\end{document}